\documentclass[12pt, reqno]{amsart}

\usepackage{amsmath,amsthm,amssymb}
\usepackage{comment}

\newtheorem{thm}{Theorem}[section]

\newtheorem{lem}[thm]{Lemma}
\newtheorem{cor}[thm]{Corollary}

\newtheorem*{mainthm*}{Theorem}

\newtheorem*{maincor*}{Corollary}

\theoremstyle{definition}

\newtheorem*{claim*}{Claim}

\newcommand{\dsp}{\displaystyle}
\newcommand{\mathotimes}{\mathbin{\overline{\otimes}}}

\newcommand{\Aut}{{\rm Aut}}
\newcommand{\Int}{{\rm Int}}

\usepackage{comment}
\usepackage{mathtools}
\usepackage{braket}

\numberwithin{equation}{section}

\allowdisplaybreaks[2]

\title[Fullness of von Neumann algebras associated with equivalence relations]{On fullness of von Neumann algebras associated with non-singular Borel equivalence relations}
\author[Y. Michimoto]{Yuta Michimoto}
\date{\today}

\begin{document}
\begin{abstract}
It is shown by Houdayer--Isono that a group measure space von Neumann algebra is a full factor 
if the group is countable discrete and bi-exact, and the action is strongly ergodic, essentially free and non-singular.
Recently, 
bi-exactness for locally compact groups was introduced by Brothier--Deprez--Vaes.
In this paper, 
we will show that  Houdayer--Isono type result holds for bi-exact locally compact groups. 
\end{abstract}
\maketitle
\section{Introduction}
Fullness of von Neumann algebras was introduced by Connes in \cite{Con74}. 
We recall that a von Neumann algebra $M$ is called \emph{full} 
if the inner automorphism group $\Int (M)$ is closed in the automorphism group $\Aut (M)$ with respect to the $u$-topology. 
For any factor $M$ with separable predual, 
$M$ is full if and only if every norm bounded central  sequence of $M$ is trivial.
For example, 
the non-abelian free group factor $L(\mathbb{F}_n)$ is full, and
non-type I amenable factors are never full.
More generally, 
Effors's result (\cite{Ef73}) shows that the group von Neumann algebra
$L(\Gamma)$ is a full factor for any non-inner amenable countably infinite group $\Gamma$. 
On group measure space von Neumann algebras,
Choda showed in \cite{Ch81} that 
the group measure space von Neumann algebra $L^{\infty}(X) \rtimes \Gamma$ is a full factor 
of type I\hspace{-.1em}I$_1$
for every strongly ergodic essentially free probability measure preserving action 
$\Gamma \curvearrowright (X, \mu)$ 
of every non-inner amenable countably infinite group $\Gamma$. 
Recall that a non-singular action is called \emph{strongly ergodic} 
if every almost invariant sequence of Borel subsets is trivial. 
In \cite[Theorem C]{HI15}, 
Houdayer--Isono showed that 
$L^{\infty}(X) \rtimes \Gamma$ is a full factor (can be of type I\hspace{-.1em}I\hspace{-.1em}I)
for every strongly ergodic essentially free non-singular action 
$\Gamma \curvearrowright (X, \mu)$ 
of every bi-exact (i.e. exact and has property (S) ) countably infinite group $\Gamma$ on a non-atomic probability measure space $(X, \mu)$. 
Recently, 
property (S) was introduced to locally compact groups 
by Brothier--Deprez--Vaes (\cite{BDV18}). 
Then, the following question naturally arises.
Does Houdayer--Isono type result hold for locally compact groups having property (S)?
In this paper, we give an affirmative answer to this question.
For this proof,
we will use a cross section equivalence relation $\mathcal{R}_1$
associated to the action $G \curvearrowright (X, \mu)$.
Here, 
it is important to use property (S) of an equivalence relation introduced by Deprez (\cite{Dep19}).
Therefore, 
the aim is to show the following statement
(Theorem \ref{mainthm}).
\begin{mainthm*}
Let $\mathcal{R}$ be a non-singular countable Borel equivalence relation 
on a non-atomic standard probability space. 
If 
$\mathcal{R}$
is strongly ergodic
and
has property (S),
then
$L(\mathcal{R})$ is a full factor.
\end{mainthm*}
If we note that the strong ergodicity of the original action is inherited by a cross section equivalence relation
(Lemma \ref{strongly ergodic}), 
we obtain the following corollary (Corollary \ref{cormain}).
\begin{maincor*}
Let $G$ be a locally compact group which has property (S) and 
$G\curvearrowright (X, \mu)$
a strongly ergodic essentially free non-singular action
on a non-atomic standard probability space.
Then the crossed product von Neumann algebra 
$L^{\infty}(X, \mu) \rtimes G$
is a full factor. 
\end{maincor*}

Throughout this paper, locally compact groups are assumed to be Hausdorff and second countable,
and measure spaces are assumed to be non-atomic and standard.
\section{Preliminaries}
\subsection{von Neumann algebras}
In this subsection,
we will prepare notations for von Neumann algebras.
For any von Neumann algebra $M$,
we will denote by $M^{U}$ and $M^{P}$ the set of unitaries and projections in $M$, respectively.
In this paper, 
von Neumann algebras are assumed to be separable, i.e. with separable predual. 
For every faithful state $\varphi \in M_{\ast}$,
we write $\|x\|_{\varphi}:= \varphi(x^*x)^{\frac{1}{2}}$.

Let $M$ be a von Neumann algebra and $\omega \in \beta(\mathbb{N})\setminus \mathbb{N}$ be any
non-principal ultrafilter.
We denote by ${\ell}^{\infty}(M)$ the C$^*$-algebra of norm bounded sequences in $M$.
Define 
\begin{align*}
\mathcal{I}_{\omega}(M) &\coloneqq 
\Set{(x_n)_n \in {\ell}^{\infty}(M)| x_n \to 0 \ast\text{-strongly as } n \to \omega}, \\
\mathcal{M}^{\omega}(M) &\coloneqq
\Set{(x_n)_n \in {\ell}^{\infty}(M)| (x_n)_n\mathcal{I}_{\omega}(M) \subset \mathcal{I}_{\omega}(M) \text{ and }
\mathcal{I}_{\omega}(M)(x_n)_n \subset \mathcal{I}_{\omega}(M)}, \\
\mathcal{C}_{\omega}(M) &\coloneqq
\Set{(x_n)_n \in {\ell}^{\infty}(M)| \lim_{n \to \omega} \|\varphi(x_n \cdot) - \varphi(\cdot x_n)\|=0 
\text{ for all } \varphi \in M_{\ast}}.
\end{align*}
The subalgebras $\mathcal{M}^{\omega}(M)$ and $\mathcal{C}_{\omega}(M)$
are unital C$^*$-subalgebras of ${\ell}^{\infty}(M)$
and 
$\mathcal{I}_{\omega}(M)$ is a norm closed two-sided ideal of $\mathcal{M}^{\omega}(M)$ and $\mathcal{C}_{\omega}(M)$.
Then
$\mathcal{I}_{\omega}(M) \subset \mathcal{C}_{\omega}(M) \subset \mathcal{M}^{\omega}(M)$.
We define the \emph{ultraproduct von Neumann algebras} $M^{\omega}$ and $M_{\omega}$ by
the quotient C$^*$-algebras 
$M^{\omega}\coloneqq \mathcal{M}^{\omega}(M)/\mathcal{I}_{\omega}(M)$
and 
$M_{\omega}\coloneqq \mathcal{C}_{\omega}(M)/\mathcal{I}_{\omega}(M)$, respectively.
Indeed, 
$M^{\omega}$ and $M_{\omega}$ are known to be von Neumann algebras.
We have $M_{\omega} \subset M^{\omega}$.
Note that $M$ is regard as a von Neumann subalgebra of $M^{\omega}$ by 
the mapping 
$M \ni x \mapsto (x, x, ...) \in M^{\omega}$.
See \cite{AH14} and \cite{Oc85} for more details.
\subsection{Non-singular Borel equivalence relations}
Let $(X, \mu)$ be a standard measure space,
and $\mathcal{R} \subset X \times X$ an equivalence relation. 
If $\mathcal{R}$ is a Borel subset in $X\times X$ 
and each equivalence class of $\mathcal{R}$ is a countable set,
then $\mathcal{R}$ is called \emph{countable Borel}. 
The equivalence class of $x \in X$ is denoted by 
$[x]_{\mathcal{R}}\coloneqq \Set{y \in X| (x, y) \in \mathcal{R}}$.
We say that $\mathcal{R}$ is \emph{non-singular} for $\mu$
if for all Borel subsets $E \subset X$ with $\mu(E)=0$,
we have $\mu([E]_{\mathcal{R}})=0$,
where $[E]_{\mathcal{R}}:= \Set{x \in X | (x, y) \in \mathcal{R}\  \text{for some}\  y \in E }$.
In this paper, 
an equivalence relation always means countable Borel and non-singular.  
A Borel subset $\mathcal{W} \subset \mathcal{R}$ is called \emph{bounded} if
there exists $M >0$ such that 
$\#_x\mathcal{W} < M$ and $\#\mathcal{W}_y < M$
for $\text{a.e.~} x , y\in X$ 
where
$_x\mathcal{W} \coloneqq  \Set{y \in X| (x, y) \in \mathcal{W}}$
and
$\mathcal{W}_y \coloneqq  \Set{x \in X| (x, y) \in \mathcal{W}}$.
We also call that 
a Borel subset $\mathcal{W} \subset \mathcal{R}$ is called \emph{locally bounded}
if for every $\delta >0$ there exists a Borel subset $E \subset X$ with 
$\mu(X \setminus E) < \delta$ 
such that 
$\mathcal{W} \cap (E \times E)$ is bounded.
The \emph{full group} $[\mathcal{R}]$ of $\mathcal{R}$ 
is the group of all Borel automorphisms 
$s \colon X \to X$
such that
$\mathrm{graph}(s)\coloneqq \Set{(s(x), x)|x \in X}$ 
is contained in $\mathcal{R}$.
The \emph{full pseudogroup} $[[\mathcal{R}]]$ is the set of all partial Borel automorphisms 
$s \colon A \to B$ for Borel sets $A, B \subset X$ 
such that $\mathrm{graph}(s)$ is contained in $\mathcal{R}$. 
Then we write $\mathrm{Dom}(s) \coloneqq A$ and $\mathrm{Im}(s) \coloneqq B$.

We will recall the so-called Krieger construction (\cite{FM75}).
Let $\mathcal{R}$ be a countable Borel equivalence relation on a standard measure space $(X, \mu)$.
Then, we can get the left counting measure $\mu_{l}$ and the right counting measure $\mu_{r}$ on $\mathcal{R}$ as follows:
\begin{align*}
&\mu_l(A) \coloneqq \int_{X} \#\Set{y \in X| (x, y) \in A} \, \mathrm{d}\mu(x), \\
&\mu_r(A) \coloneqq \int_{X} \#\Set{x \in X| (x, y) \in A} \, \mathrm{d}\mu(y)
\end{align*}
for each Borel subset $A \subset \mathcal{R}$.
If $\mathcal{R}$ is non-singular for $\mu$,
then $\mu_l$ is equivalent to $\mu_r$.
A function $F \in L^{\infty}(\mathcal{R}, \mu_l)$ is called \emph{left bounded}
if its support is a bounded Borel subset of $\mathcal{R}$.
The set of left bounded functions on $\mathcal{R}$ is written as $\mathcal{M}_f(\mathcal{R})$. 
We equip $\mathcal{M}_f(\mathcal{R})$ with the operation defined by
\begin{align*}
(\lambda_1F_1 + \lambda_2F_2)(x, y) &\coloneqq \lambda_1F_1(x, y) + \lambda_2F_2(x, y), \\
(F_1 \ast F_2)(x, y) &\coloneqq \sum_{z \in [x]_{\mathcal{R}}} F_1(x, z)F_2(z, y), \\
F_1^{\ast}(x, y) &\coloneqq \overline{F_1(y, x)}
\end{align*}
for $F_1, F_2 \in \mathcal{M}_f(\mathcal{R})$ and $\lambda_1, \lambda_2 \in \mathbb{C}$.
Then we regard $\mathcal{M}_f(\mathcal{R})$ as a $\ast$-subalgebra of $\mathbf{B}(L^2(\mathcal{R}, \mu_l))$
by the map 
$L_F \colon L^2(\mathcal{R}, \mu_l) \to L^2(\mathcal{R}, \mu_l)$,
where 
$L_F$ is defined by
\[L_F(\xi)(x, y):= \sum_{z \in [x]_{\mathcal{R}}} F(x, z) \xi(z, y). \]
Thus, we can obtain the von Neumann algebra as follows:
\begin{align*}
L(\mathcal{R}) \coloneqq \Set{L_F| F\in \mathcal{M}_f(\mathcal{R})}'' \subset \mathbf{B}(L^2(\mathcal{R}, \mu_l)).
\end{align*}
The subalgebra $\mathcal{M}_f(\mathcal{R}) \subset L(\mathcal{R})$ contains a canonical copy of $L^{\infty}(X)$ 
by identifying an $f \in L^{\infty}(X)$ 
with the operator associated to the left bounded function 
$\mathcal{R} \ni (x, y) \mapsto f(x) 1_{\Delta}(x, y)$, 
where $1_{\Delta}$ denotes the characteristic function of the diagonal set 
$\Delta \coloneqq \Set{(x, y) \in X \times X | x=y} \subset \mathcal{R}$ .
If we define operators $L_f, u_s$ on $L^2(\mathcal{R})$ by 
\[
(L_f \xi) (x, y) := f(x) \xi(x, y)
\hspace{3mm}
\text{and}
\hspace{3mm}
(u_{s} \xi) (x, y) := \xi(s^{-1}(x), y) 1_{\mathrm{Im}(s)} (x)
\]
for all $f \in L^{\infty}(X)$,
$s \in [[\mathcal{R}]]$,
$\xi \in L^2(\mathcal{R})$
and
$\text{a.e.~} (x, y) \in \mathcal{R}$,
then $L(\mathcal{R})$ is generated by $L^{\infty}(X)$ and $\{u_s\}_{s \in [\mathcal{R}]}$.
Note that 
$u_{s} L_f u_s^{\ast} = L_{f \circ s^{-1}}$
for $f \in L^{\infty}(X)$ and $s \in [\mathcal{R}]$.
For every essentially free non-singular action $\Gamma \curvearrowright X$ of a countable group $\Gamma$,
we have the $\ast$-isomorphism
\[
L(\mathcal{R}(\Gamma \curvearrowright X))
\cong
L^{\infty}(X) \rtimes \Gamma, 
\]
where 
$\mathcal{R}(\Gamma \curvearrowright X)$
denotes the orbit equivalence relation of
$\Gamma \curvearrowright X$.
We remark that there exists a faithful normal conditional expectation $E$ from $L(\mathcal{R})$ to $L^{\infty}(X)$,
and $L^{\infty}(X)$ is a Cartan subalgebra of $L(\mathcal{R})$.
The von Neumann algebra $L(\mathcal{R})$ is a factor if and only if $\mathcal{R}$ is ergodic.
See \cite{FM75} and \cite[Section 2]{Ko} for more details.

We will recall the so-called Zimmer's amenability. 
We say that 
a equivalence relation $\mathcal{R}$ on $(X, \mu)$ is \emph{amenable}
if there exists a unital completely positive (u.c.p.) map 
$\Phi \colon L^{\infty}(\mathcal{R}, \mu_l) \to L^{\infty}(X, \mu)$
such that 
$\Phi(F^{s}) = \Phi(F)^s$ ($F \in L^{\infty}(\mathcal{R}, \mu_l)$) for every partial transformation $s \in [[\mathcal{R}]]$,
where 
$F^{s}, \; F \in L^{\infty}(\mathcal{R}, \mu_l)$
and 
$f^{s}, \; f \in L^{\infty}(X, \mu)$
are respectively defined as follows 
$F^{s}(x, y)
:= F(s^{-1}(x), y) 1_{\mathrm{Im}(s)}(x)$
and
$f^{s}(x)
:=f(s^{-1}(x)) 1_{\mathrm{Im}(s)}(x)$.
This function $\Phi$
is called a 
\textit{left invariant mean}
 on $\mathcal{R}$.
We recall that a von Neumann algebra $L(\mathcal{R})$ is amenable if and only if
an equivalence relation $\mathcal{R}$ is amenable.
See \cite{Zim77} for more details.
\subsection{Cross section equivalence relations}\label{subsection_cross_section}
For a non-singular action $G \curvearrowright (X, \mu)$ of a locally compact (non-countable) group,
the orbit equivalence relation 
$\mathcal{R}(G \curvearrowright (X, \mu))$
is uncountable in general.
Moreover, 
there is no faithful normal conditional expectation from $L^{\infty}(X) \rtimes G$ to $L^{\infty}(X)$.
A cross section defined by Forrest (\cite{For74}) is a useful tool to overcome these difficulties. 
See \cite{KPV15} for the content of this subsection.

Let $G \curvearrowright (X, \mu)$ be an essentially free non-singular action of a locally compact group $G$.
A \emph{cross section} $X_1$ for $G \curvearrowright (X, \mu)$ is a Borel subset of $(X, \mu)$ satisfying the following properties.
\begin{enumerate}
\item
There exists a neighborhood $U \subset G$ of the neutral element $e\in G$
such that the action map
$U \times X_1 \ni (g, x) \mapsto g \cdot x \in X$
is injective ;
\item
The Borel subset $G\cdot X_1 \subset X$ is conull,
i.e. $\mu(X \setminus G\cdot X_1)=0$.
\end{enumerate}
Forrest showed in \cite[Proposition 2.10]{For74} that 
there is always a cross section for every essentially free non-singular action 
$G \curvearrowright (X, \mu)$ of a locally compact group $G$.
By definition of a cross section,
a equivalence relation 
\begin{align*}
\mathcal{R}_1 := \Set{(x, y) \in X_1\times X_1|y=g\cdot x  \text{ for some } g \in G}
\end{align*}
is countable Borel on $X_1$
(see \cite[Proposition 4.3]{KPV15}).
It is called the \emph{cross section equivalence relation} for $G \curvearrowright (X, \mu)$. 

It is known that much of the information of 
the action $G \curvearrowright (X, \mu)$ 
is inherited by a cross section equivalence relation.
Let $G \curvearrowright (X, \mu)$ be an essentially free non-singular action.
Let $X_1 \subset X$ be a cross section for $G \curvearrowright (X, \mu)$
and denote by $\mathcal{R}_1$ its cross section equivalence relation.
Then there exists a unique $\sigma$-finite measure $\mu_1$ on $X_1$ satisfying
\begin{equation}\label{equation_cross_measure}
(\lambda_G \otimes \mu_1)(W) = \int_{X} \sum_{\substack{(g, y) \in W\\ x=gy}} D(g^{-1}, x) \, \mathrm{d}\mu(x)
\end{equation}
for every measurable subset $W \subset G \times X_1$,
where $\lambda_G$ is the left Haar measure on $G$ and $D(g, x)$ is the Radon--Nikodym derivatives
(see \cite[Proposition 4.3]{KPV15} and \cite[Proposition 2.5.38]{De}).
Moreover, $\mathcal{R}_1$ is non-singular for $\mu_1$.
Since the action map 
$G \times X_1 \ni (g, x) \mapsto g \cdot x \in X$
is countable-to-one and essentially surjective,
it admits a Borel right inverse 
$X \ni x \mapsto (\gamma(x), \pi(x)) \in G\times X_1$.
Using the Borel map $\pi \colon X \to X_1$,
we can see that the ergodicity of $G \curvearrowright (X, \mu)$ is equivalent to the ergodicity of $\mathcal{R}_1$.
Then,
it is shown in \cite[Lemma 4.5]{KPV15} between the crossed product von Neumann algebra $L^{\infty}(X) \rtimes G$ and the von Neumann algebra $L(\mathcal{R}_1)$,
that
\begin{align*}
p(L^{\infty}(X)\rtimes G)p \cong L(\mathcal{R}_1) \mathotimes \mathbf{B}(L^2(U))
\end{align*}
where $p:= 1_{U \cdot X_1} \in L^{\infty}(X)^P$.
\subsection{Property (S)}\label{subsection_(S)}
Property (S) for countable groups was introduced in \cite{Oz06}.
Recently, 
property (S) was introduced for locally compact groups in \cite{BDV18}.
We say that
a locally compact group $G$ has \emph{property (S)} 
if there exists a $\|\cdot\|_1$-continuous map 
$m \colon G \ni g \mapsto m_g \in \rm{Prob}(G)$
such that
$\lim_{h \to \infty}\|m_{ghk}- g_{\ast}m_h\|_1 = 0$
uniformly on compact subsets for $g, k \in G$,
where
$\rm{Prob}(G)$ denote the set of probability measures on $G$.
We also call that $G$ is \emph{bi-exact} 
if $G$ is exact in the sense of Kirchberg--Wassermann \cite{KW99}
and has property (S).
For example, 
free groups $\mathbb{F}_n$, 
$\mathrm{SL}(2, \mathbb{R})$,
$\mathbb{R}^2 \rtimes \mathrm{SL}(2, \mathbb{R})$
and 
the groups with metrically proper actions on trees
satisfy property (S),
and $\mathrm{SL}(3, \mathbb{R})$ does not satisfy property (S)
(see \cite[Proposition 7.1]{BDV18} and \cite{Dep19}).

In \cite{Dep19}, 
Deprez introduced property (S) for non-singular Borel equivalence relations $\mathcal{R}$ as well.
We say that an equivalence relation $\mathcal{R}$ has \emph{property (S)} 
if there exist functions $\eta(x, y) \in {\ell}^1([x]_{\mathcal{R}})$
for $(x, y) \in \mathcal{R}$
with $\sum_{z \in [x]_{\mathcal{R}}} \eta(x, y)_{(z)}=1$,
such that 
for all $s, t \in [\mathcal{R}]$ and $\varepsilon >0$
the subset 
\begin{align*}
\Set{(x, y) \in \mathcal{R}| \|\eta(s(x), t(y))- \eta(x, y)\|_{{\ell}^1([x]_{\mathcal{R}})} \geq \varepsilon
}
\end{align*}
of $\mathcal{R}$
is locally bounded.
Then the following result \cite[Lemma 5.4]{Dep19} holds.
\begin{lem}[Deprez]\label{lem_Deprez}
Let $\mathcal{R}$ be an ergodic non-singular countable Borel equivalence relation on a measure space $(X, \mu)$
and let $X_0$ be a non-negligible subset of $(X, \mu)$. 
Then, 
$\mathcal{R}$ has property (S) if and only if the restriction
$\mathcal{R} \cap (X_0 \times X_0)$ has property (S).
\end{lem}
\subsection{Full factors and strongly ergodic actions}
We say that a norm bounded sequence $(x_n)_n \subset M$
is \emph{central} if 
$\lim_{n\to \infty} \|\varphi(x_n\, \cdot )- \varphi( \cdot\, x_n) \|=0$
for all $\varphi \in M_{\ast}$,
and $(x_n)_n$ is \emph{trivial} if
there exits a bounded sequence $(\lambda_n)_n \subset \mathbb{C}$ 
such that 
$x_n - \lambda_n 1_{M} \to 0$
with respect to the $\ast$-strong topology. 
It is known that a factor $M$ is full if and only if
every central sequence in $M$ is trivial. 
Also, note that a factor $M$ is full if and only if $M_{\omega}=\mathbb{C}1$ 
for some (or any) non-principal ultrafilter $\omega \in \beta(\mathbb{N}) \setminus \mathbb{N}$.
By \cite[Lemma 2.11]{Con74},
for a separable von Neumann algebra $M$ 
and a separable type I factor $Q$,
we have $(M \mathotimes Q)_{\omega} \cong M_{\omega} \mathotimes \mathbb{C}1$. 
By \cite[Lemma 2.10]{MT16},
for a separable von Neumann algebra $M$ and any projection $p \in M^{P}$,
we have $(pMp)_{\omega} \cong pM_{\omega}p$. 
In summary, we obtain the following well-known result.
\begin{lem}\label{lem_ultra}
Let $M$ be a separable factor, $Q$ a separable type I factor and $p \in M^{P}$.
Then the fullness of $M$, $M \mathotimes Q$ and $pMp$ are equivalent.
\end{lem}
Strong ergodicity is the concept corresponding to fullness for non-singular actions. 
Let $G \curvearrowright (X, \mu)$ be a non-singular action 
($G$ may  be a locally compact group)
and the induced action $\sigma \colon G \curvearrowright L^{\infty}(X, \mu)$
be given by 
$\sigma_g(f)(x) \coloneqq f(g^{-1} x)$
($g \in G, f \in L^{\infty}(X), x \in X$).
We say that a sequence $(f_n)_n \subset L^{\infty}(X)$ is \emph{almost $G$-invariant} if
$\sigma_g(f_n) - f_n \to 0$
in the measure topology,
uniformly on compact subsets for $g\in G$,
and $(f_n)_n \subset L^{\infty}(X)$ is \emph{trivial} if
there exists a bounded sequence $(\lambda_n)_n$ in $\mathbb{C}$ 
such that $f_n - \lambda_n 1 \to 0$  in the measure topology.
We recall that a non-singular action is \emph{strongly ergodic} 
if every almost invariant sequence is trivial. 
By definition, 
strongly ergodicity implies ergodicity. 
By Rokhlin lemma, 
every ergodic essentially free non-singular action
$\mathbb{Z} \curvearrowright (X, \mu)$
on a non-atomic measure space is not strongly ergodic. 
Recall that,
for a non-singular Borel equivalence relation $\mathcal{R}$ on a measure space $(X, \mu)$, 
a sequence $(f_n)_n \subset L^{\infty}(X)$ is \emph{almost $\mathcal{R}$-invariant} if 
$f^l_n - f^r_n \to 0$ in the measure topology,
where 
$f^l_n(x, y)= f_n(x), f^r_n(x, y)= f_n(y) \text{ for } (x, y) \in \mathcal{R}$,
and 
$\mathcal{R}$ is \emph{strongly ergodic} if every almost $\mathcal{R}$-invariant sequence is trivial.
We remark that
all amenable non-singular equivalence relations on a non-atomic measure space are never strongly ergodic.
Note that 
the ergodic non-singular equivalence relation $\mathcal{R}$ on $(X, \mu)$ is strongly ergodic 
if and only if 
$L(\mathcal{R})'\cap L^{\infty}(X, \mu)^{\omega} = \mathbb{C}$
for some (or any) non-principal ultrafilter $\omega \in \beta(\mathbb{N})\setminus \mathbb{N}$.
If $L(\mathcal{R})$ is full, then $\mathcal{R}$ is strongly ergodic.  
It has been pointed out by Connes--Jones \cite{CJ81} that the converse statement is not true.
Also, 
Ueda showed that 
non-singular Borel equivalence relations of type I\hspace{-.1em}I\hspace{-.1em}I$_0$ are not strongly ergodic \cite[Corollary 11]{Ue00}. 
\section{Proof of the main theorem}
For a probability measure preserving action $G \curvearrowright (X, \mu)$,
Deprez showed that
property (S) of  the locally compact group $G$ is equivalent to
property (S) of a cross section equivalence relation $\mathcal{R}_1$ for $G \curvearrowright (X, \mu)$.
Based on this proof, 
we can see that one direction also holds for a non-singular action.
\begin{lem}[cf.{\cite[Proposition 5.5]{Dep19}}]\label{property (S)}
Let $G$ be a locally compact group and 
$G\curvearrowright (X, \mu)$
an ergodic essentially free
non-singular action.
Let $X_1 \subset X$ be a cross section and 
$\mathcal{R}_1$ the associated cross section equivalence relation. 
If
$G$ has property (S), 
then 
so does $\mathcal{R}_1$. 
\end{lem}
\begin{proof}
Fix Borel maps 
$\gamma \colon X \to G$
and
$\pi \colon X \to X_1$
with 
$x = \gamma(x) \pi(x)$
for
$\text{a.e.~} x \in X$
as in Subsection \ref{subsection_cross_section}.
Take a map 
$\eta \colon G \ni g \mapsto \eta_g \in \mathrm{Prob}(G)$
as in the definition of property (S) for $G$.
For each $x \in X$,
we set the map 
$\pi_x \colon G \ni g \mapsto \pi(g^{-1}x) \in X_1$. 
Fix the cocycle $\omega \colon \mathcal{R}_1 \to G$ satisfying $\omega(x, y)y= x$ for all $(x, y) \in \mathcal{R}_1$.
We define functions $\eta_1(x, y) \in \ell^{1}([x]_{\mathcal{R}_1})$ by
\begin{equation*}
\eta_1 (x, y) := (\pi_x)_{\ast}\eta_{\omega(x, y)}
\end{equation*}
for $(x, y) \in \mathcal{R}_1$.
Note that every function $\eta_1(x, y)$ is a probability measure on the $\mathcal{R}_1$-orbit of $x$.
Indeed, 
\[\eta_1(x, y)([x]_{\mathcal{R}_1})
= \eta_{\omega(x, y)}(\pi_x^{-1}([x]_{\mathcal{R}_1}))
= \eta_{\omega(x, y)} (G)
=1. \]

We now,  take a Borel subset $X_0 \subset X_1$ such that $0 < \mu_1(X_0) < +\infty$.
Since $G \curvearrowright (X, \mu)$ is ergodic, 
$\mathcal{R}_1$ is ergodic. 
By ergodicity of $\mathcal{R}_1$, 
we have a partition $X_1 = \bigcup_{i\in I} X_i$ up to measure zero and 
Borel isometries $\phi_i \in [\mathcal{R}_1]$ with $\phi_i(X_i) \subset X_0$.
Set 
$\mathcal{R}_0 := \mathcal{R}_1 \cap (X_0 \times X_0)$. 
By Lemma \ref{lem_Deprez}, 
in order to prove that $\mathcal{R}_1$ has property (S),
it suffices to show that $\mathcal{R}_0$ has property (S). 
We define,
for $(x, y) \in \mathcal{R}_0$,
the probability measure $\eta_0(x, y)$
on the $\mathcal{R}_0$-orbit of $x$ 
by setting 
\begin{equation*}\dsp
\eta_0(x, y)(z) := \sum_{i\in I_z}\eta_1(x, y)(\phi_i^{-1}(z)) 
\end{equation*}
for $z \in [x]_{\mathcal{R}_0}$,
where $I_z:= \{i \in I\mid z\in \phi_i(X_i)\}$.

Fix $\varepsilon>0$ and $s, t \in [\mathcal{R}_0]$. 
Since $G$ is second countable, 
we can take a family of  compact subsets $\{K_n\}_{n \in \mathbb{N}}$ in $G$ with  
$G =\bigcup_nK_n$. 
We set,
for $n \in \mathbb{N}$, 
$E_n := \{x \in X_0 \mid \omega(s(x), x) \in K_n \}$, 
and we have $X_0 = \bigcup_n E_n$.
Since $\lim_n \mu_1(X_0 \setminus E_n) =0$
(note that 
$0 <\mu_1(X_0) < +\infty$), 
by taking $n_0$ large enough,
we obtain 
$\mu_1(X_0 \setminus E_{n_0}) < \varepsilon$.
Using this argument, 
we can find a compact set $K \subset G$ and a measurable set $E \subset X_0$ 
with
$\mu_1(X_0 \setminus E) < \varepsilon$
such that 
$\omega(s(x), x) \in K$ and $\omega(y, t(y)) \in K$
for all $x, y \in E$.
Take a compact set $L \subset G$ such that 
$\|\eta_{gkh} - g \cdot \eta_k\|_1 < \varepsilon$
for all $g, h\in K$ and all $k \in G \setminus L$.
\begin{claim*}
For $(x, y)\in \mathcal{R}_0 \cap (E \times E)$, 
$
\|\eta_0(s(x), t(y))- \eta_0(x, y)\|_{{\ell}^{1}([x]_{\mathcal{R}_0})} < \varepsilon
$
whenever $\omega(x, y) \in G \setminus L$.
\end{claim*}
\begin{proof}[Proof of claim]
Fix $(x, y)\in \mathcal{R}_0$
with
$\omega(x, y) \in G \setminus L$. 
Now, note that
\[\pi_x(g)=
\pi(g^{-1}x)
= \pi(g^{-1}\omega(s(x), x)^{-1} s(x))
= \pi((\omega(s(x), x)g)^{-1}s(x))
= \pi_{s(x)}(\omega(s(x), x)g)
\]
and hence 
$
(\pi_x)_{\ast}\eta_{\omega(x, y)}
= (\pi_{s(x)})_{\ast}(\omega(s(x), x) \cdot \eta_{\omega(x, y)}).
$
Then we have
\begin{align*}\dsp
&\hspace{5mm}\|\eta_0(s(x), t(y))- \eta_0(x, y)\|_{{\ell}^{1}([x]_{\mathcal{R}_0})} \\
&= \sum_{
z \in [x]_{\mathcal{R}_0}} 
\left|
\sum_{i \in I_z}\eta_1(s(x), t(y))(\phi_i^{-1}(z))- \eta_1(x, y)(\phi_i^{-1}(z))
\right|\\
&= \sum_{
z \in [x]_{\mathcal{R}_0}} 
\left|
\sum_{i \in I_z}(\pi_{s(x)})_{\ast}\eta_{\omega(s(x), t(y))}(\phi_i^{-1}(z))
- (\pi_{s(x)})_{\ast}(\omega(s(x), x) \cdot \eta_{\omega(x, y)})(\phi_i^{-1}(z))
\right|\\
&\leq \|(\pi_{s(x)})_{\ast}\eta_{\omega(s(x), t(y))}-(\pi_{s(x)})_{\ast}(\omega(s(x), x) \cdot \eta_{\omega(x, y)}) \|_1 \\
&\leq
\|\eta_{\omega(s(x), t(y))}-
\omega(s(x), x) \cdot \eta_{\omega(x, y)}\|_1 \\
&=
\|\eta_{\omega(s(x),x) \omega(x, y) \omega(y, t(y))} -\omega(s(x), x) \cdot \eta_{\omega(x, y)} \|_1
< \varepsilon,
\end{align*}
in the last inequality we used the assumption that
$\omega(s(x), x) \in K$,
$\omega(y, t(y)) \in K$
and 
$\omega(x, y) \in G \setminus L$.
Thus,  the claim is proved. 
\end{proof}
\noindent
By the above claim,
we obtain the inclusion 
\begin{align*}
&\{(x, y)\in \mathcal{R}_0\cap(E \times E) \mid \|\eta_0(s(x), t(y))- \eta_0(x, y)\|_{{\ell}^{1}([x]_{\mathcal{R}_0})} \geq \varepsilon\}\\
\subset&
\{(x, y)\in \mathcal{R}_0\cap(E \times E) \mid \omega(x, y)\in L\}.
\end{align*}
Since the latter set is bounded by compactness of $L$,
the set $\{(x, y)\in \mathcal{R}_0 \mid \|\eta_0(s(x), t(y))- \eta_0(x, y)\|_{{\ell}^{1}([x]_{\mathcal{R}_0})} \geq \varepsilon\}$
is locally bounded. 
Thus, $\mathcal{R}_0$ has property (S), and so does $\mathcal{R}_1$ by Lemma \ref{lem_Deprez}.
\end{proof}
\begin{lem}\label{strongly ergodic}
Let $G$ be a locally compact group and 
$G\curvearrowright (X, \mu)$
an essentially free 
non-singular action.
Let $X_1 \subset X$ be a cross section and 
$\mathcal{R}_1$ the associated cross section equivalence relation. 
If $G\curvearrowright (X, \mu)$ is strongly ergodic, 
then $\mathcal{R}_1$ is strongly ergodic. 
\end{lem}
\begin{proof}
Take a neighborhood $U$ of the neutral element $e \in G$ 
as in the definition of the cross section $X_1$ for $G\curvearrowright (X, \mu)$.
Fix Borel maps 
$\gamma \colon X \to G$
and
$\pi \colon X \to X_1$
with 
$x = \gamma(x) \pi(x)$
for
$\text{a.e.~} x \in X$
as before.
Take a Borel subset $X_0$ in $(X_1, \mu_1)$ 
with $0 < \mu_1(X_0) < + \infty$.
Set $\mathcal{R}_0 \coloneqq \mathcal{R}_1 \cap (X_0 \times X_0)$.

First, we show that $\mathcal{R}_0$ is strongly ergodic on $(X_0, \mu_1)$.
Let 
$(f_n)_n \subset L^{\infty}(X_0, \mu_1)$
be any almost $\mathcal{R}_0$-invariant sequence. 
Since $G\curvearrowright (X, \mu)$ is ergodic by the strong ergodicity of $G\curvearrowright (X, \mu)$,
$\mathcal{R}_1$ is ergodic.
Then we choose a partition $X_1 = \bigcup_{i\in I} X_i$ up to measure zero and
Borel isometries $\phi_i \in [\mathcal{R}_1]$ with $\phi_i(X_i) \subset X_0$ and $\phi_i(\omega) = \omega$ for all $\omega \in X_0$.
For $x \in X$, 
there exists a unique $i \in I$ such that $\pi(x) \in X_i$,
and we have $\phi_i(\pi(x)) \in X_0$.
For each $n$, 
we define a map $F_n \in L^{\infty}(X, \mu)$ by $F_n(x) := f_n(\phi_i(\pi(x)))$ ($x \in X$). 
Then $(F_n)_n \subset L^{\infty}(X, \mu)$ is an almost $G$-invariant sequence.
In fact,
for each $x \in X$ and $g \in G$,
there exist $i, j \in I$ such that 
$\phi_i(\pi(x)) \in X_0$ and $\phi_j(\pi(gx)) \in X_0$.
Since $\phi_i, \phi_j \in [\mathcal{R}_1]$, we have $(\phi_i(\pi(x)), \phi_j(\pi(gx))) \in \mathcal{R}_0$.
Thus, $(F_n)_n$ is almost $G$-invariant on $(X, \mu)$.
It follows from our assumption that $(F_n)_n$ is a trivial sequence on $(X, \mu)$.
We prove that $(f_n)_n$ is a trivial sequence on $(X_0, \mu_1)$.
Fix $\varepsilon > 0$.
Then there exists a sequence $(\lambda_n)_n \subset \mathbb{C}$ 
such that $F_n - \lambda_n \to 0$ in the measure topology on $(X, \mu)$.
Set 
$Y_n \coloneqq \Set{x \in X|\, |F_n(x)-\lambda_n| > \varepsilon}$
and 
$Z_n \coloneqq \Set{\omega \in X_0|\, |f_n(\omega)-\lambda_n| > \varepsilon}$.
Now, note that
\begin{align}\label{equation_trivial_sequence_1}
&\int_{X} 1_{UX_0}(x) D(\gamma(x)^{-1}, x) \, \mathrm{d}\mu(x)
= \lambda_{G}(U) \mu_1(X_0) 
< + \infty, \\
&\int_{X} 1_{UZ_n}(x) D(\gamma(x)^{-1}, x) \, \mathrm{d}\mu(x)
=\lambda_{G}(U) \mu_1(Z_n) \leq \lambda_{G}(U) \mu_1(X_0) 
< + \infty 
\end{align}
by the equation (\ref{equation_cross_measure}) and the injectivity of the action map 
$U \times X_1 \ni (u, \omega) \mapsto u \cdot \omega \in X$.
Since $\lim_{n \to \infty} \mu(Y_n)=0$ and $U\cdot Z_n \subset Y_n$,
$\lim_{n \to \infty} \mu(U\cdot Z_n)= 0$.
Thus, 
$\lim_{n \to \infty} \mu_1(Z_n) =0$
by the equations \eqref{equation_trivial_sequence_1}, (3.2) and 
the absolute continuity of 
$E \mapsto \int_{X} 1_{E} 1_{UX_0} D(\gamma(x)^{-1}, x)\, \mathrm{d}\mu(x)$.
It follows that $(f_n)_n \subset L^{\infty}(X_0, \mu_1)$ is a trivial sequence.
Hence 
$\mathcal{R}_0$
is strongly ergodic on $(X_0, \mu_1)$.

Second,
we show that
$\mathcal{R}_1$ is strongly ergodic on $(X_1, \mu_1)$.
In general, 
an ergodic equivalence relation $\mathcal{R}$ on $(X, \mu)$ is strongly ergodic 
if and only if 
$L(\mathcal{R})' \cap L^{\infty}(X, \mu)^{\omega} = \mathbb{C}$. 
Moreover, 
by \cite[Lemma 2.10]{MT16},
for a separable von Neumann algebra $M$ and any projection $p \in M^{P}$,
we have $(pMp)^{\omega} \cong pM^{\omega}p$. 
Now, put $p \coloneqq 1_{X_0} \in L^{\infty}(X_1, \mu_1)^{P}$.
Since 
$\mathcal{R}_0$ is strongly ergodic on $(X_0, \mu_1)$,
$L(\mathcal{R}_0) = pL(\mathcal{R}_1)p$ and $L^{\infty}(X_0, \mu_1) = L^{\infty}(X_1, \mu_1)p$,
we have $L(\mathcal{R}_1)' \cap L^{\infty}(X_1, \mu_1)^{\omega} = \mathbb{C}$.
It follows that a cross section equivalence relation $\mathcal{R}_1$ is strongly ergodic on $(X_1, \mu_1)$. 
\end{proof}
\begin{thm}\label{mainthm}
Let $\mathcal{R}$ be any non-singular countable Borel equivalence relation 
on a non-atomic standard probability space $(X, \mu)$. 
If 
$\mathcal{R}$
is strongly ergodic
and
has property (S),
then
$L(\mathcal{R})$ is a full factor.
\end{thm}
\begin{proof}
By contradiction,
assume that the factor $L(\mathcal{R})$ is not full.
We will write $M \coloneqq L(\mathcal{R})$.
We set 
$A \coloneqq L^{\infty}(X, \mu) \subset M$
and 
denote by
$E_A \colon M \to A$ the unique faithful normal conditional expectation. 
Fix any faithful state $\tau \in A_{\ast}$ and put 
$\varphi \coloneqq \tau \circ E_A \in M_{\ast}$.
Since $\mathcal{R}$ is strongly ergodic and $M$ is not full,
by \cite[Lemma 5.1]{HI15},
there exists a central sequence of unitaries 
$(u_n)_n \subset M^{U}$
such that
$\lim_{n\to\infty}\|E_A(xu_ny)\|_{\tau}=0$
for all $x, y \in M$.

Modifying the method of \cite[Example 8]{Oz16},
we will construct a left invariant mean 
$\Phi \colon L^{\infty}(\mathcal{R}, \mu_l) \to L^{\infty}(X, \mu)$
on $\mathcal{R}$.
Since $\mathcal{R}$ has property (S),
there exist functions $\eta(x, y) \in {\ell}^1([x]_{\mathcal{R}})$
for $(x, y) \in \mathcal{R}$
as in Subsection \ref{subsection_(S)}. 
Take a measurable function $F_n$ on $\mathcal{R}$ so that  
\begin{align*}
(u_n\xi)(x, y)= \sum_{z \in [x]_{\mathcal{R}}}
F_n(x, z) \xi(z, y)
\end{align*}
for
$\xi \in L^2(\mathcal{R}, \mu_l)$
and 
$(x, y) \in \mathcal{R}$.
Then 
$\sum_{z \in [x]_{\mathcal{R}}}|F_n(x, z)|^2=1$.
For each $n$,
we define a map
$\Phi_n \colon L^{\infty}(\mathcal{R}, \mu_{l}) \to L^{\infty}(X, \mu)$
by 
\begin{equation}\label{equation_goal_ucp}
\dsp
\Phi_n(f)(x) 
:=
\sum_{y, z \in [x]_{\mathcal{R}}} |F_n(x, z)|^2\eta(x, z)_{(y)}f(x, y)
\end{equation}
for 
$f \in L^{\infty}(\mathcal{R}, \mu_{l})$
and
$x \in X$.
Note that $\Phi_n$ is a u.c.p.~map. 
By Lemma \ref{keylem} below,  we get the equation
\begin{align}\label{key_equi}
\dsp\lim_{n \to \infty}\|\Phi_n(f)^{s}- \Phi_n(f^{s})\|_{L^1(X, \mu)} = 0\hspace{3mm}
\text{for all}\hspace{3mm}
 s \in [\mathcal{R}].
\end{align} 
We denote $\Phi \colon L^{\infty}(\mathcal{R}, \mu_l) \to L^{\infty}(X, \mu)$ by
the clustar point of $(\Phi_n)_n$ with respect to the
pointwise ultraweak topology.
Note that $\Phi$ is a u.c.p.~map.
By the equation (\ref{key_equi}),
we have
$\Phi(f^{s})= \Phi(f)^{s}$
for every 
$f \in L^{\infty}(X, \mu)$
and
$s \in [\mathcal{R}]$. 

Fix $s \in [[\mathcal{R}]]$.
Then there exists a partition 
$\{E_n\}_n$ 
of 
$\mathrm{Dom}(s)$
and
$\{s_n\}_n \subset [\mathcal{R]}$
such that 
$s_n(E_n)=: F_n$
are disjoint and
$s(x)= s_n(x)$
for
$x \in E_n$.
Now, we have
\begin{align*}
&1_{F_n} \Phi(f^{s})
= \Phi((1_{F_n}\circ \pi_l)f^{s})
=\Phi(((1_{E_n}\circ \pi_l)f)^{s_n})\\
=&\Phi((1_{E_n}\circ \pi_l)f)^{s_n}
=(1_{E_n}\Phi(f))^{s_n}
=1_{F_n}\Phi(f)^{s}
\end{align*}
where 
$\pi_l(x, y) = x$ for $(x, y) \in X \times X$.
Hence 
$\Phi(f^{s})=\Phi(f)^{s}$
on
$\mathrm{Im}(s).$
With
$F:= X \setminus \mathrm{Im}(s)$,
we have
$1_F \Phi(f)^{s}=0$
and
\begin{align*}
1_F\Phi(f^{s})=\Phi((1_F\circ\pi_l)f^{s})\Phi(0)=0.
\end{align*}
Hence we conclude 
$\Phi(f^{s})=\Phi(f)^{s}$
for all 
$s \in [[\mathcal{R]]}$.
This implies that $\mathcal{R}$ is amenable,
which contradicts the strong ergodicity of $\mathcal{R}$.
\end{proof}
\begin{lem}\label{keylem}
On the u.c.p.~map $\Phi_n$ defined in (\ref{equation_goal_ucp}),
we have
\[\dsp\lim_{n \to \infty}\|\Phi_n(f)^{s}- \Phi_n(f^{s})\|_{L^1(X, \mu)} = 0\]
for all $s \in [\mathcal{R}]$ and $f \in L^{\infty}(\mathcal{R}, \mu_l)$.
\begin{proof}
Fix $s \in [\mathcal{R}]$.
On the one hand, 
we have
\begin{align*}
\Phi_n(f)^{s}(x)
&= \Phi_n(f)(s^{-1}(x)) \\
&= \sum_{y, z \in [s^{-1}(x)]_{\mathcal{R}}=[x]_{\mathcal{R}}}
|F_n(s^{-1}(x), z)|^2\eta(s^{-1}(x), z)_{(y)}f(s^{-1}(x), y) .
\end{align*}
On the other hand, 
we have
\begin{align*}
\Phi_n(f^{s})(x)
&= \sum_{y , z\in [x]_{\mathcal{R}}}|F_n(x, z)|^2\eta(x, z)_{(y)}f(s^{-1}(x), y) \\
&= \sum_{y , z\in [x]_{\mathcal{R}}}|F_n(x, s(z))|^2\eta(x, s(z))_{(y)}f(s^{-1}(x), y) \\
&= \sum_{y , z\in [x]_{\mathcal{R}}}|F_n(x, s(z))|^2 \eta(s^{-1}(x), z)_{(y)} f(s^{-1}(x), y)\\
&\hspace{1cm}+
\sum_{y , z\in [x]_{\mathcal{R}}}|F_n(x, s(z))|^2 \left(\eta(x, s(z))_{(y)} - \eta(s^{-1}(x), z)_{(y)}\right)
f(s^{-1}(x), y) .
\end{align*}
This implies that
\begin{align} \dsp 
&\|\Phi_n(f)^{s}- \Phi_n(f^{s})\|_{L^1(X, \mu)} \notag\\
= 
&\int_X |\Phi_n(f)^{s}(x)- \Phi_n(f^{s})(x)| \;\mathrm{d}\mu (x) \notag\\
\leq 
&\int_X 
\sum_{y, z \in [x]_{\mathcal{R}}}
\left|
\left(|F_n(s^{-1}(x), z)|^2-|F_n(x, s(z))|^2 \right)
\eta(s^{-1}(x), z)_{(y)}f(s^{-1}(x), y)
\right| 
\;\mathrm{d}\mu (x) \tag{A}\\
&+ 
\int_X
\sum_{y , z\in [x]_{\mathcal{R}}}
\left|
|F_n(x, s(z))|^2 
\left(\eta(x, s(z))_{(y)} - \eta(s^{-1}(x), z)_{(y)}\right)
f(s^{-1}(x), y)
\right|
\;\mathrm{d}\mu (x)  \tag{B}.
\end{align}
We will compute (A).
We now have,
for $x \in X$,
\begin{align*}\dsp
&\sum_{y, z \in [x]_{\mathcal{R}}}
\left|
|F_n(s^{-1}(x), z)|^2-|F_n(x, s(z))|^2
\eta(s^{-1}(x), z)_{(y)} 
\right|  \\
=&
\sum_{z \in [x]_{\mathcal{R}}}
\Big( \sum_{y \in [x]_{\mathcal{R}}} \eta(s^{-1}(x), z)_{(y)} \Big)
\left|
|F_n(s^{-1}(x), z)|^2-|F_n(x, s(z))|^2
\right| \\
=&
\sum_{z \in [x]_{\mathcal{R}}}
\left|
|F_n(s^{-1}(x), z)|^2-|F_n(x, s(z))|^2 \right|\\
=&
\sum_{z \in [x]_{\mathcal{R}}}
\left|
\left(|F_n(s^{-1}(x), z)|-|F_n(x, s(z))|)\cdot(|F_n(s^{-1}(x), z)|+|F_n(x, s(z))|\right)
\right|\\
\leq&
\left(
\sum_{z \in [x]_{\mathcal{R}}}
\left|
|F_n(s^{-1}(x), z)|-|F_n(x, s(z))|
\right|^2
\right)^{\frac{1}{2}}
\left(\sum_{z \in [x]_{\mathcal{R}}}
\left|
|F_n(s^{-1}(x), z)|+|F_n(x, s(z))|
\right|^2
\right)^{\frac{1}{2}} \\
\leq& 
2\|(u_n u_{s}(1_{\Delta}))(x, \cdot)- (u_{s}u_n(1_{\Delta}))(x, \cdot)\|_{{\ell}^2([x]_{\mathcal{R}})}.
\end{align*}
At the last inequality,
we used the following computation:
\begin{align*}
&\sum_{z \in [x]_{\mathcal{R}}}
\left|
|F_n(s^{-1}(x), z)|-|F_n(x, s(z))|
\right|^2\\
\leq&
\sum_{z \in [x]_{\mathcal{R}}}
\left|
F_n(s^{-1}(x), z)-F_n(x, s(z))
\right|^2\\
=&
\sum_{z \in [x]_{\mathcal{R}}}
\left|
\sum_{u \in [s^{-1}(x)]_{\mathcal{R}}}
F_n(s^{-1}(x), u)1_{\Delta}(u, z)
-
\sum_{u \in [x]_{\mathcal{R}}}
F_n(x,u)1_{\Delta}(s^{-1}(u), z)
\right|^2\\
=&
\sum_{z \in [x]_{\mathcal{R}}}
\left|
u_{s}u_n(1_{\Delta})(x, z) - u_nu_{s}(1_{\Delta})(x,z)
\right|^2,
\end{align*}
and
\[
\left(\sum_{z \in [x]_{\mathcal{R}}}
\left|
|F_n(s^{-1}(x), z)|+|F_n(x, s(z))|
\right|^2
\right)^{\frac{1}{2}}
\leq 2
\]
since
$u_n \in M$ is a unitary
which satisfies
$(u_n\xi)(x, y)= \sum_{z \in [x]_{\mathcal{R}}}
F_n(x, z) \xi(z, y)$.
By the centrality of
$(u_n)_n \subset M$,
we obtain
\begin{align*}
\text{(A)}
=&\int_X 
\sum_{y, z \in [x]_{\mathcal{R}}}
\left|
(|F_n((s^{-1}(x), z))|^2-|F_n(x, s(z))|^2 )
\eta(s^{-1}(x), z)_{(y)}f(s^{-1}(x), y)
\right| 
\;\mathrm{d}\mu (x) \\
\leq&
\|f\|_{\infty}
\int_X
\sum_{y , z\in [x]_{\mathcal{R}}}
\left|
\left(|F_n((s^{-1}(x), z))|^2-|F_n(x, s(z))|^2 \right)
\eta(s^{-1}(x), z)_{(y)}
\right|\;\mathrm{d}\mu (x)\\
\leq&
2\|f\|_{\infty}
\int_{X}
\|(u_n u_{s}(1_{\Delta}))(x, \cdot)- (u_{s}u_n(1_{\Delta}))(x, \cdot)\|_{{\ell}^2([x]_{\mathcal{R}})} \, \mathrm{d}\mu(x)\\
\leq&
2\|f\|_{\infty}
\left(
\int_{X}
\|(u_n u_{s}(1_{\Delta}))(x, \cdot)- (u_{s}u_n(1_{\Delta}))(x, \cdot)\|^2_{{\ell}^2([x]_{\mathcal{R}})} \, \mathrm{d}\mu(x)
\right)^{\frac{1}{2}}\\
=&
2\|f\|_{\infty}
\|u_nu_s(1_{\Delta}) - u_su_n(1_{\Delta})\|_{\tau}^{\frac{1}{2}}
\to 0
\hspace{5mm}
(n \to \infty).
\end{align*}
Next, we will compute (B).
Fix 
$\varepsilon>0$.
We set 
$\mathcal{W}
:=
\{(x, z) \in \mathcal{R}
\mid 
\|\eta(x, s(z)) - \eta(s^{-1}(x), z)\|_1
\geq \varepsilon\}
\subset 
\mathcal{R}$.
From the condition of $\eta$,
$\mathcal{W}$ 
is a locally bounded subset of 
$\mathcal{R}$.
Now, since $\mathcal{R}$ is non-singular,
take $0 < \delta$ ($< \varepsilon$) so that
$\mu(s(U)) < \varepsilon$
for every measurable subset $U \subset X$ with $\mu(U) < \delta$.
There exists 
a measurable subset
$E \subset X$ 
with
$\mu(X \setminus E) < \delta$
such that
$\mathcal{V}:= \mathcal{W} \cap (E \times E)$ is bounded.
Now, we prove the following claim.
\begin{claim*}
If we take $n_0 \in \mathbb{N}$ large enough,
then
we have
\[
\int_{\mathcal{W}} |F_{n}(x, s(z))|^2 \;\mathrm{d}\mu_l (x, z) < 5 \varepsilon,
\]
for all $n \geq n_0$.
\begin{proof}[Proof of claim]
Set $F\coloneqq X \setminus E$.
Since $(u_n)_n$ is a central sequence, 
we can find $n_1 \in \mathbb{N}$ large enough so that with
\[\left|\|1_{\Delta \cap (s(F)\times s(F))} u_{n_1}^{\ast} 1_{\Delta \cap (E \times E)} \|_{\varphi}^2-
\|u_{n_1}^{\ast} 1_{\Delta \cap (s(F)\times s(F))} 1_{\Delta \cap (E \times E)} \|_{\varphi}^2\right|< \varepsilon. \]
Note that 
for measurable subsets $E_1, E_2 \subset X$ and every $n \in \mathbb{N}$,
\begin{align*}
&\|1_{\Delta \cap (s(E_2)\times s(E_2))} u_n^{*} 1_{\Delta \cap (E_1 \times E_1)} \|_{\varphi}^2
= \int_{\mathcal{R} \cap (E_1 \times E_2)} |F_n(x, s(z))|^2  \;\mathrm{d}\mu_l (x, z) \\
&\|u_{n}^{\ast} 1_{\Delta \cap (s(E_2)\times s(E_2))} 1_{\Delta \cap (E_1\times E_1)} \|_{\varphi}^2
= \|1_{\Delta \cap (s(E_2)\times s(E_2))} 1_{\Delta \cap (E_1 \times E_1)} \|_{\tau}^2
\leq 
\mu(s(E_2)).
\end{align*}
Then we have
\begin{align}\label{ineq_1}
&\int_{\mathcal{W}\cap (E\times F)} |F_{n_1}(x, s(z))|^2 \;\mathrm{d}\mu_l (x, z)\notag\\
\leq&
\left|\int_{\mathcal{W}\cap (E\times F)} |F_{n_1}(x, s(z))|^2 \;\mathrm{d}\mu_l (x, z)
-\|1_{\Delta \cap (s(F)\times s(F))} 1_{\Delta \cap (E \times E)} \|_{\tau}^2\right| \notag\\
&+
\|1_{\Delta \cap (s(F)\times s(F))} 1_{\Delta \cap (E \times E)} \|_{\tau}^2
\notag\\
\leq&
\left|\int_{\mathcal{W}\cap (E\times F)} |F_{n_1}(x, s(z))|^2 \;\mathrm{d}\mu_l (x, z)
- \|1_{\Delta \cap (s(F)\times s(F))} 1_{\Delta \cap (E \times E)} \|_{\tau}^2\right|
+ \mu(s(F))\notag\\
=&
\left|\|1_{\Delta \cap (s(F)\times s(F))} u_{n_1}^{\ast} 1_{\Delta \cap (E \times E)} \|_{\varphi}^2-
\|u_{n_1}^{\ast} 1_{\Delta \cap (s(F)\times s(F))} 1_{\Delta \cap (E \times E)} \|_{\varphi}^2\right|
+ \mu(s(X \setminus E))\notag \\
<& \varepsilon +\varepsilon = 2\varepsilon.
\end{align}
Since $u_n$ is a unitary in $M$ which satisfies 
$u_n(\xi)(x, y)= \sum_{z \in [x]_{\mathcal{R}}}F_n(x, z) \xi(z, y)$,
we have
\begin{align}\label{ineq_2}
&\int_{\mathcal{W}\cap (F\times E)} |F_n(x, s(z))|^2 \;\mathrm{d}\mu_l (x, z) \notag
= \int_{F} \Big(\sum_{z \in [x]_{\mathcal{R}}\cap E}\left|F_n(x, s(z)) \right|^2 \Big) \, \mathrm{d}\mu(x)\\
\leq&
\mu(F)
= \mu(X \setminus E)
<\delta < \varepsilon
\end{align}
for all $n \in \mathbb{N}$.
Similarly, we get 
\begin{align}\label{ineq_3}
\int_{\mathcal{W}\cap (F\times F)} |F_n(x, s(z))|^2 \;\mathrm{d}\mu_l (x, z) < \varepsilon
\end{align}
for all $n \in \mathbb{N}$.
Now, 
we choose
$s_1, ..., s_N \in [[\mathcal{R}]]$
such that 
$\mathcal{V} \subset 
\bigcup_{i=1}^N \mathrm{graph}(s_i)$
by the boundedness of $\mathcal{V}$.
Since 
$\lim_n\|E_A(xu_ny)\|_{\tau}=0$
for all $x, y \in M$,
we can find $n_2 \in \mathbb{N}$ large enough so that with 
$\sum_{i=1}^N \|E_A(u_{n_2} u_s u_{s_i})\|_{\tau}^2 <\varepsilon$.
Then
\begin{align}\label{ineq_4}
&\int_{\mathcal{W}\cap(E \times E)} |F_{n_2}(x, s(z))|^2 \;\mathrm{d}\mu_l (x, z)
\leq
\sum_{i=1}^N
\int_{\mathrm{graph}(s_i)} |F_{n_2}(x, s(z))|^2 \;\mathrm{d}\mu_l (x, z)
\notag
\\
=&
\sum_{i=1}^N
\int_X
|F_{n_2}(x, s(s_i(x)))|^2 \;\mathrm{d}\mu(x)
=
\sum_{i=1}^N
\|E_A(u_{n_2} u_s u_{s_i})\|_{\tau}^2
< \varepsilon.
\end{align}
Set $n_0 \coloneqq \max\{n_1, n_2\}$.
By the inequalities (\ref{ineq_1}) -- (\ref{ineq_4}),
we obtain the following inequality, 
\[
\int_{\mathcal{W}} |F_{n_0}(x, s(z))|^2 \;\mathrm{d}\mu_l (x, z) < 5 \varepsilon.
\]
\end{proof}
\end{claim*}
Using the above claim,
we have the following, for all $n \geq n_0$,
\begin{align*}
\text{(B)}=
&\int_X
\sum_{y , z\in [x]_{\mathcal{R}}}
\left|
|F_n(x, s(z))|^2 (\eta(x, s(z))_{(y)} - \eta(s^{-1}(x), z)_{(y)})
f(s^{-1}(x), y)
\right|
\;\mathrm{d}\mu (x)\\
\leq&
\|f\|_{\infty}
\int_X 
\sum_{z \in [x]_{\mathcal{R}}}
|F_n(x, s(z))|^2
\|\eta(x, s(z)) - \eta(s^{-1}(x), z)\|_{{\ell}^{1}([x]_{\mathcal{R}})}
\;\mathrm{d}\mu (x)\\
=&
\|f\|_{\infty}
\int_{\mathcal{R}}
|F_n(x, s(z))|^2
\|\eta(x, s(z)) - \eta(s^{-1}(x), z)\|_{{\ell}^{1}([x]_{\mathcal{R}})}
\;\mathrm{d}\mu_l (x, z)\\
=&
\|f\|_{\infty}
\int_{\mathcal{R\setminus \mathcal{W}}}
|F_n(x, s(z))|^2
\|\eta(x, s(z)) - \eta(s^{-1}(x), z)\|_{{\ell}^{1}([x]_{\mathcal{R}})}
\;\mathrm{d}\mu_l (x, z)\\
&\hspace{1cm}+
\|f\|_{\infty}
\int_{\mathcal{W}}
|F_n(x, s(z))|^2
\|\eta(x, s(z)) - \eta(s^{-1}(x), z)\|_{{\ell}^{1}([x]_{\mathcal{R}})}
\;\mathrm{d}\mu_l (x, z)\\
\leq&
\|f\|_{\infty} \varepsilon 
+
2 \|f\|_{\infty}
\int_{\mathcal{W}} |F_n(x, s(z))|^2 \;\mathrm{d}\mu_l (x, z)\\
\leq&
 \|f\|_{\infty} \varepsilon + 10 \|f\|_{\infty} \varepsilon
= 11\|f\|_{\infty} \varepsilon. 
\end{align*}
Thus, 
we conclude
$\lim_{n \to \infty}\|\Phi_n(f)^{s}- \Phi_n(f^s)\|_{L^1(X)}=0$.
\end{proof}
\end{lem}

\begin{cor}\label{cormain}
Let $G$ be a locally compact group which has property (S) and 
$G\curvearrowright (X, \mu)$
a strongly ergodic essentially free
non-singular action
on a non-atomic standard probability space.
Then the crossed product von Neumann algebra 
$L^{\infty}(X, \mu) \rtimes G$
is a full factor. 
\end{cor}
\begin{proof}
Take a cross section $X_1 \subset X$ and
$\mathcal{R}_1$ the associated cross equivalence relation. 
Let $p:= 1_{U\cdot X_1}$ be a projection in $L^{\infty}(X, \mu)$, 
where $U$ is a neighborhood of the neutral element of $G$ such that the action map 
$U \times X_1 \to X$ is injective.
Thanks to  \cite[Lemma 4.5]{KPV15},
\begin{equation*}
p(L^{\infty}(X) \rtimes G)p \cong L(\mathcal{R}_1) \mathotimes \mathbf{B}(L^2(U)).
\end{equation*}
Thus,
it suffices to show that the factor $L(\mathcal{R}_1)$ is full by Lemma \ref{lem_ultra}. 
Take a non-negligible subset
$Y \subset X_1$
with finite measure for $\mu_1$.
We denote by 
the probability measure $\nu:= \mu_1(Y)^{-1}\mu_1$ on $Y$
and 
the countable equivalence relation 
$\mathcal{S}:= \mathcal{R}_1 \cap (Y \times Y)$
on
$(Y, \nu)$.
By Lemma \ref{property (S)} and Lemma \ref{strongly ergodic}, 
$\mathcal{R}_1$ has property (S) and is strongly ergodic,
hence
the relation $\mathcal{S}$
has property (S) by Lemma \ref{lem_Deprez}
and
is strongly ergodic by using \cite[Lemma 2.10]{MT16}.
Then,
it follows that $L(\mathcal{S})= qL(\mathcal{R}_1)q$ is a full factor by Theorem \ref{mainthm},
where $q\coloneqq 1_Y \in L^{\infty}(X_1, \mu_1)^{P}$.
Thus,
$L(\mathcal{R}_1)$ is full by Lemma \ref{lem_ultra},
hence the theorem is proved. 
\end{proof}
\subsection*{Acknowledgements}
The author would like to express his deep gratitude to his supervisor, 
Professor Reiji Tomatsu for his support and providing many insightful comments.
He is also grateful to Professor Yusuke Isono and  Professor Yoshimichi Ueda for helpful comments.
This work was supported by the Research Institute for Mathematical Sciences, an International Joint
Usage/Research Center located in Kyoto University.

\end{document}